\documentclass[letterpaper, 10 pt, conference]{ieeeconf}  

\IEEEoverridecommandlockouts                              
\usepackage{bbm}
\usepackage{mathrsfs}
\usepackage{verbatim}
\usepackage{amsmath}
\usepackage{amsfonts}
\usepackage{bm}
\usepackage{graphicx}
\usepackage{caption}
\usepackage{float}

\usepackage{amsthm}
\usepackage{xcolor}
\usepackage[colorinlistoftodos]{todonotes}
\usepackage{mathtools}
\usepackage{cite}

\theoremstyle{definition}

\newtheorem{definition}{Definition}
\newtheorem{remark}{Remark}
\newtheorem{problem}{Problem}

\theoremstyle{plain}

\newtheorem{theorem}{Theorem}
\newtheorem{lemma}{Lemma}

\title{\LARGE \bf
Decentralized Control of Two Agents with Nested Accessible Information}

\author{Aditya Dave, {\itshape{Student Member, IEEE,}} Nishanth Venkatesh, {\itshape{Student Member, IEEE,}} \\  Andreas A. Malikopoulos, {\itshape{Senior Member, IEEE}} 
	\thanks{This research was supported by the Sociotechnical Systems Center (SSC) at the University of Delaware.} %
	\thanks{The authors are with the Department of Mechanical Engineering, University of Delaware, Newark, DE 19716 USA (email: \texttt{adidave@udel.edu; nish@udel.edu; andreas@udel.edu).}} }

\begin{document}

\maketitle
\thispagestyle{empty}

\begin{abstract}
In this paper, we investigate a decentralized stochastic control problem with two agents, where a part of the memory of the second agent is also available to the first agent at each instance of time. We derive a structural form for optimal control strategies which allows us to restrict their domain to a set which does not grow in size with time. We also present a dynamic programming (DP) decomposition which can utilize our results to derive optimal strategies for arbitrarily long time horizons. Since obtaining optimal control strategies by solving this DP decomposition is computationally intensive, we present potential resolutions in the form of simplified strategies by imposing additional conditions on our model, and an approximation technique which can be used to implement our results with a bounded loss of optimality.
\end{abstract}

\section{Introduction} \label{section:Introduction}

Decentralized stochastic control problems consist of cooperative agents who take actions over a time horizon to minimize a shared cost, with limited ability to communicate in real time. Typical decentralized systems include connected and automated vehicles \cite{malikopoulos2021optimal} and social media platforms \cite{Dave2020SocialMedia}. 
Generally, no single agent has both: (1) access to all information in the system and (2) the ability to assign all actions. Thus, such systems are characterized by their \textit{information structure}, which describes the information available to each agent at each time. Various information structures, summarized in \cite{mahajan2012}, are:
\textit{(1) Classical,}  where all agents communicate and recall information perfectly \cite{varaiya_book};
\textit{(2) Quasi-classical,} if agent $1$ can affect the state of agent $2$, and the information available to agent $1$ is \textit{also} available to agent $2$ \cite{lessard2013structural}; and
\textit{(3) Non-classical,} where agents can affect each others' states with incomplete information \cite{17, mahajan2013controlsharing, nayyar2011structure, Dave2019a, Dave2021a, xie2020optimally, dave2020structural, Malikopoulos2021}. 

Non-classical systems suffer from doubly exponential growth in computations required to generate optimal control strategies with an increase in the planning horizon \cite{bernstein2002complexity}. The common information approach \cite{17} alleviates this problem for systems with partial history sharing among all agents. The main idea in this approach is to identify an \textit{information state} which can be utilized, in place of the common information across all agents, to derive the optimal control strategies. For systems with partial history sharing, the information state and private information of each agent do not grow in size with time. Subsequently, we can use a dynamic programming (DP) decomposition to compute optimal control strategies for long time horizons. However, the computational tractability of this approach suffers if the private information of any agent grows with time. This phenomenon is commonly observed in systems with partially nested information \cite{lessard2013structural}, one-directional communication \cite{mahajan2013controlsharing, nayyar2011structure, Dave2019a, Dave2021a} and unreliable communication \cite{asghari2018optimal, ouyang2016optimal}. For such systems, current methods focus on identifying specific dynamics and information structures which yield computationally tractable solutions \cite{mahajan2013controlsharing, nayyar2011structure}. 

In this paper, we identify a general information structure, called \textit{nested accessible information,} for decentralized systems with two agents, and show that even in the presence of noisy observations of the state, it yields control strategies which are functions of information states. At each instance of time, we consider that a subset of the information available to agent $2$, called accessible information, is sequentially nested within the information available to agent $1$. However, the information which is available to agent $1$ and not available to agent $2$ is allowed to grow in size with time. For example, this phenomenon occurs when agent $1$ does not share their observations and actions with agent $2$ but receives the actions and observations of agent $2$ at each time. Other special cases of our information structure include teams of two agents with: (1) either instantaneous or delayed one-directional communication from agent $2$ to agent $1$ \cite{Dave2021a, xie2020optimally}; (2) transmission of data from agent $1$ to agent $2$ using an unreliable communication channel, which was considered for systems with linear dynamics, quadratic costs and Gaussian noises in \cite{ouyang2016optimal}; (3) real time communication from agent $1$ to agent $2$ \cite{nayyar2011structure}. 


Our main contribution in this paper is that we establish a structural form for optimal control strategies in systems with nested accessible information (Theorem \ref{pbp_struct_result_2}). This structural form allows us to restrict their domain to a space which does not grow in size with time. In our exposition, we use a combination of the \textit{person-by-person} \cite{lessard2013structural} and \textit{prescription} \cite{dave2020structural} approaches. While both approaches are well established, we combine them to yield results for optimal strategies which cannot be derived by an individual application of either of these approaches. 
Next, we present a DP decomposition which utilizes our results to obtain optimal control strategies (Section \ref{subsection:DP}). In general, solving this DP is computationally challenging. As a potential resolution, we show how our results can be simplified with an additional assumption of decoupled dynamics for the system (Section \ref{section:decoupled}). Finally, we propose an approximate solution which can be used to improve the computational tractability of the DP even in the presence of coupled dynamics (Section \ref{section:Implementation}). While we restrict our attention to a team of two agents to simplify the exposition, our results can also be applied to systems with multiple agents in two nested subsystems, using a technique presented in Section III of \cite{Dave2021b}.


The remainder of the paper is organized as follows. In Section II, we provide our problem formulation. In Section III, we analyze the problem and derive our main results. In Section IV, we present specialized results for systems with additional assumptions on the dynamics. In Section V, we present an approximation technique to implement our results. Finally, in Section VI, we present concluding remarks and discuss ongoing work.

\section{Problem Formulation}
\label{section:problem}
We consider a team of two agents who take actions over $T \in \mathbb{N}$ discrete time steps. For each $t=0,\dots, T$, the state of the team is denoted by the random variable $X_t$ which takes values in a finite set $\mathcal{X}_t$. 
The action of an agent $k = 1,2$ at time $t$ is $U_t^k$, which takes values in a finite set $\mathcal{U}^k_t$. We denote the tuple $(U_t^1,U_t^2)$ by ${U}_t^{1:2}$. Starting at the initial state $X_0$, the system evolves as
\begin{equation}
    X_{t+1}=f_t\left(X_t,U_t^{1:2},W_t\right), \quad t = 0,\dots,T-1, \label{st_eq}
\end{equation}
where $W_t$ is an uncontrolled disturbance which takes values in a finite set $\mathcal{W}_t$. At each $t = 0,\dots,T$, each agent $k = 1,2$ makes an observation $Y_t^k := h_t^k(X_t,V_t^k)$,
which takes values in a finite set 
$\mathcal{Y}^k_t$. Here, $V_t^k$ 
is a measurement noise which takes values in a finite set $\mathcal{V}^k_t$.
The external disturbances $\{W_t: t=0,\dots,T\}$, measurement noises $\{V_t^1, V_t^2: t=0,\dots,T\}$, and initial state $X_0$ are collectively called the \textit{primitive random variables} of the team and their probability distributions are known a priori.
We assume that each primitive random variable is independent of all other primitive random variables to ensure that the system's evolution is Markovian \cite{varaiya_book}.

\begin{definition}
For all $t=0,\dots,T$, the \textit{memory} of an agent $k = 1,2$ is a set of random variables $M_t^k \subseteq \{Y_{0:t}^{1:2}, U_ {0:t-1}^{1:2}\}$, which takes values in a finite collection of sets $\mathcal{M}_t^k$ and satisfies \textit{perfect recall}, i.e, $M_{t-1}^k \subseteq M_{t}^k$, with $M_{-1}^k := \emptyset$.
\end{definition}

We partition the memory $M_t^2$ of agent $2$ into two components, the \textit{accessible information} $A_t^2$ and \textit{private information} $L_t^2$, which are described next:

\textit{1) The accessible information} is a subset of the memory of agent $2$ which is also available to agent $1$. For all $t=0,\dots,T,$ we define the accessible information as a set of random variables $A_t^2 \subseteq M_t^2$ which takes values in a finite collection of sets $\mathcal{A}_t^2$ and satisfies the properties: (1) accessibility to agent 1, i.e., $A_t^2 \subseteq M_t^1$, and (2) perfect recall, i.e., $A_{t-1}^2 \subseteq A_{t}^2$, with $A_{-1}^2 := \emptyset$.

\textit{2) The private information} of agent $2$ is a subset of their memory which is unavailable to agent $1$. For all $t=0,\dots,T$, we define the private information as the set of random variables $L_t^2 := M_t^2 \setminus A_t^2$ which takes values in a finite collection of sets $\mathcal{L}_t^2$. We impose the condition $L_t^2 \cap M_t^1 = \emptyset$ to specify that agent $1$ can not access the private information of agent $2$ at each $t$.

The second property of the accessible information of agent $2$ motivates us to define the \textit{new information} added to $A_t^2$, for all $t=0,\dots,T$, as the set of random variables $Z_{t}^2 := A_{t}^2 \setminus A_{t-1}^2$ which takes values in a finite collection of sets $\mathcal{Z}_t^2$. Note that $Z_0^2 := A_0^2$. Analogously, for all $t=0,\dots,T$, we define the new information added to the memory of agent $1$ as the set of random variables $Z_t^1 := M_t^1 \setminus M_{t-1}^1$ which takes values in a finite collection of sets $\mathcal{Z}_t^1$, where $Z_0^1 := M_0^1$. In our information structure, we enforce that for all $t$, the new information of agent $2$ must satisfy $Z_t^2 \subseteq L_t^2 \cup \{Y_t^{1:2}, U_{t-1}^{1:2}\}$. This ensures that $Z_t^2 \not\subset M_{t-1}^1$ and $Z_t^2 \subseteq Z_t^1$, i.e., $Z_t^2$ is not accessible to agent $1$ prior to time $t$ and becomes accessible to agent $1$ at time $t$.

\begin{remark}
We call the shared set $A_t^2$ the \textit{accessible information} of agent $2$ instead of \textit{common information} \cite{17} to highlight the additional restriction imposed by the property $Z_t^2 \not\subset M_{t-1}^1$. The presence of this restriction allows us to specialize our results to systems where the information available to agent $1$ but unavailable to agent $2$, i.e., $M_t^1 \setminus A_t^2$, may grow in size with time. If we relax this restriction, the accessible information is equivalent to common information. 
\end{remark}

\begin{remark}
As an example of an information structure which satisfies $Z_t^2 \not\subset M_{t-1}^1$, consider one-directional communication from $2$ to $1$ with a delay of $d \in \mathbb{N}$ time steps. In such a system, $M_t^1 = \{Y_{0:t}^{1},  U_{0:t-1}^{1}, Y_{0:t-d}^{2}, U_{0:t-d}^2\}$ and $M_t^2 = \{Y_{0:t}^2, U_{0:t-1}^{2}\}$. Then, $A_t^2 = \{Y_{0:t-d}^{2}, U_{0:t-d}^2\}$, $L_t^2 = \{Y_{t-d+1:t}^2, U_{t-d+1:t-1}^2\}$, and the set $M_t^1 \setminus A_t^2 = \{Y_{0:t}^1, U_{0:t-1}^1\}$ grows in size with time. Recall that we have referenced other information structures which satisfy the conditions for nested accessible information in Section \ref{section:Introduction}.
\end{remark}


For all $t=0,\dots,T$, each agent $k = 1,2$ uses a control law $g_t^k: \mathcal{M}_t^k \to \mathcal{U}_t^k$ to select their action
\begin{gather}
    U_t^k = g_t^k(M_t^k),
\end{gather}
where $M_t^2 = \{L_t^2, A_t^2\}$.
We define the control strategy of agent $k$ as $\boldsymbol{g}^k := (g_t^k: t=0,\dots,T)$ and the control strategy of the team as $\boldsymbol{g} := (\boldsymbol{g}^1,\boldsymbol{g}^2)$. The set of all feasible control strategies is $\mathcal{G}$.
After each agent $k = 1,2$ selects their action $U_t^k$ at time $t$, the team incurs a cost $c_t(X_t,U_t^{1:2}) \in \mathbb{R}_{\geq0}$. The performance criterion over the finite horizon $T$ is
\begin{equation}
    \mathcal{J}(\boldsymbol{g}) = \mathbb{E}^{\boldsymbol{g}}\left[\sum_{t=0}^T{c_t\big(X_t,U_t^{1:2}\big)}\right], \label{per_cri}
\end{equation}
where 
the expectation is with respect to the joint probability distribution on all random variables. 
Next, we state the optimization problem for the team.

\begin{problem} \label{problem_1}
The optimization problem for the team is $\inf_{\boldsymbol{g} \in \mathcal{G}} \mathcal{J}(\boldsymbol{g})$, given the distributions of the primitive random variables $\{X_0,W_{0:t},V_{0:t}^{1:2}\}$, and the dynamics $\{c_t,f_t,h_t^{1:2}:t=0,\dots,T\}$.
\end{problem}

Problem \ref{problem_1} is guaranteed to have a solution because all variables take values in finite sets.
Our goal is to derive a structural form for an optimal strategy $\boldsymbol{g}^* \in \mathcal{G}$ in Problem \ref{problem_1} which can be computed using a DP decomposition.

\section{Analysis Using Prescriptions} \label{section:chain}


\subsection{Analysis for Agent 1} \label{pbp_agent_1}

In this subsection, we derive a structural form for an optimal control strategy of agent $1$. We first note that given a strategy $\boldsymbol{g}^2$, agent $1$ cannot generate the action $U_t^2$ for each $t$ because they cannot access the complete memory $M_t^2 = \{L_t^2, A_t^2\}$. However, they can access the component $A_t^2$. This motivates us to consider a two stage process for the generation of the action of agent $2$: (1) agent $1$ generates a prescription for agent $2$ using only $A_t^2$, and (2) agent $2$ computes $U_t^2$ using this prescription and their private information $L_t^2$.


\begin{definition}
For all $t=0,\dots,T$, a \textit{prescription} for agent $2$ is a mapping $\Gamma_t^{2}: \mathcal{L}^{2}_t \to \mathcal{U}_t^2$ which takes values in a finite set $\mathcal{F}^{2}_t$.
\end{definition}

At each $t$, the prescription for agent $2$ is generated using a prescription law $\psi_t^{2}: \mathcal{A}_t^2 \to \mathcal{F}_t^{2}$, which yields $\Gamma_t^{2} = \psi_t^{2}(A_t^2)$. We call $\boldsymbol{\psi}^2 := (\psi_t^{2} : t=0,\dots,T)$ the prescription strategy for agent $2$. Given a prescription $\Gamma_t^{2}$, the action of agent $2$ is computed as $U_t^2 = \Gamma_t^{2}\big(L_t^2\big)$.
Next, we use the person-by-person approach to set up a ``new" centralized problem for agent $1$. We proceed by arbitrarily fixing the prescription strategy $\boldsymbol{\psi}^2$ for agent $2$. Since the prescription $\Gamma_t^2$ is generated using only the accessible information $A_t^2 \subseteq M_t^1$, agent $1$ can derive the prescription using the fixed strategy as $\Gamma_t^2 = \psi_t^2(A_t^2)$. Then, we define a new state for agent $1$ as
$S_t^1 := \{X_t, L_t^2, A_t^2\}$ for all $t$, which takes values in a finite collection of sets ${\mathcal{S}}_t^1$. Given a prescription strategy $\boldsymbol{\psi}^2$, we can construct a state evolution function $\bar{f}^{{1}}_t(\cdot)$, such that ${S}^{{1}}_{t+1} = \bar{f}^{{1}}_{t}({S}^{{1}}_t, U_t^{{1}}, W_{t}, V_{t+1}^{1:2})$ and an observation rule $\bar{h}_t^1(\cdot)$ which yields $Z_{t+1}^1 = \bar{h}_t^1({S}^{{1}}_t, U_t^{{1}}, W_{t}, V_{t+1}^{1:2})$ for all $t = 0,\dots,T-1$. The existence of these functions can be verified using the dynamics and information structure of the system to write the LHS in terms of the variables in the RHS. Similarly, we can construct a cost function $\bar{c}^{{1}}_t(\cdot)$ which yields the cost $\bar{c}^{{1}}_t({S}^{{1}}_t,U^{{1}}_t) := c_t(X_t, U_t^1, \psi_t^2(A_t^2)(L_t^2))$ for all $t$. Then, for a given prescription strategy $\boldsymbol{\psi}^{2}$, the new centralized problem for agent $1$ has state ${S}_t^1$, control action $U_t^1$, observation $Z_{t}^1$, and cost $\bar{c}_t^1(S_t^1,U_t^1)$ at time $t$. 
Furthermore, the performance criterion 
is $\mathcal{J}^1(\boldsymbol{g}^{1}) := \mathbb{E}^{\boldsymbol{g}^1} [\sum_{t=0}^T \bar{c}^1_t({S}_t^1, U_t^{1})]$.

\begin{problem} \label{problem_2}
The problem for agent $1$ is
    $\inf_{\boldsymbol{g}^{1}} \mathcal{J}^1(\boldsymbol{g}^{1})$,
given a prescription strategy $\boldsymbol{\psi}^{2}$, the probability distributions of the primitive random variables $\{X_0,W_{0:t},V_{0:t}^{1:2}\}$, and the dynamics $\{\bar{c}^1_t,\bar{f}^1_t,\bar{h}_t^{1}:t=0,\dots,T\}$.
\end{problem}


\begin{lemma} \label{lem_psi_g_relation} 
For a given control strategy $\boldsymbol{g}^2$, consider a prescription strategy $\boldsymbol{\psi}^2$ such that
\begin{gather}
    \psi_t^{2}(A_t^2)(\cdot) := g_t^2(\cdot, A_t^2), \quad t = 0,\dots,T. \label{u_presc}
\end{gather}
Then, $\mathcal{J}(\boldsymbol{g}^1, \boldsymbol{g}^2) = \mathcal{J}^1(\boldsymbol{g}^1)$ for the fixed prescription strategy $\boldsymbol{\psi}^2$. Moreover, for any given prescription strategy $\boldsymbol{\psi}^2$, consider a control strategy $\boldsymbol{g}^2$ constructed as
\begin{gather}
     g_t^{2}(\cdot,A_t^{2}) := \psi_t^{2}(A_t^2)(\cdot), \quad t = 0,\dots,T. \label{u_presc_inv}
\end{gather}
Then, $\mathcal{J}^1(\boldsymbol{g}^1)$ after fixing $\boldsymbol{\psi}^2$ is equal to $\mathcal{J}(\boldsymbol{g}^1, \boldsymbol{g}^2)$.
\end{lemma}

\begin{proof}
For the first part, given a control strategy $\boldsymbol{g}$ and prescription strategy $\boldsymbol{\psi}^2$, note that $U_t^2 = g_t^{2}(L_t^2,A_t^2) = \psi_t^2(A_t^2)(L_t^{2})$, i.e., the control law and prescription law result in the same control action $U_t^2$ for a given memory $M_t^2 = \{L_t^2, A_t^2\}$, for all $t = 0,\dots,T$. Thus, after fixing $\boldsymbol{\psi}^2$, we can write the expected cost at each $t$ as $\mathbb{E}^{\boldsymbol{g}}[c_t(X_t, U_t^{1:2})] = \mathbb{E}^{\boldsymbol{g}^1}[{c}_t(X_t, U_t^1, \psi_t^2(A_t^2)(L_t^2))] = \mathbb{E}^{\boldsymbol{g}^1}[\bar{c}_t^1(S_t^1, U_t^1)]$, where the second equality holds using the construction of $\bar{c}_t^1(\cdot)$. The proof is complete by summing the cost over all time steps. For the second part, the proof follows from similar arguments as in the first part.
\end{proof}

\begin{remark} \label{remark_3}
We consider that a control strategy $\boldsymbol{g}^2$ and a prescription strategy $\boldsymbol{\psi}^2$ are always selected to satisfy \eqref{u_presc} and \eqref{u_presc_inv} simultaneously. Thus, fixing $\boldsymbol{\psi}^2$ in Problem \ref{problem_2} also fixes $\boldsymbol{g}^2$, and vice versa. Next, consider a control strategy $(\boldsymbol{g}^{*1}, \boldsymbol{g}^{*2})$ which is an optimal solution to Problem \ref{problem_2}. We construct a prescription strategy for agent $2$ as $\psi_t^{*2}(A_t^2)(\cdot) := g_t^{*2}(\cdot, A_t^2)$, for all $t =0,\dots,T$, and use the first part of Lemma \ref{lem_psi_g_relation} to conclude that $\boldsymbol{g}^{*1}$ must an optimal solution for Problem \ref{problem_2} after fixing $\boldsymbol{\psi}^{*2}$. Thus, every optimal solution to Problem \ref{problem_1} yields a corresponding solution to Problem \ref{problem_2}.
\end{remark}

Problem \ref{problem_2} is a centralized stochastic control problem for agent $1$, with a perfectly observed component $A_t^2$ of the state ${S}_t^1$ and a partially observed component $\{X_t, L_t^2\}$, which must be estimated using the memory $M_t^1$. For such an estimation problem, it is known \cite[page 79]{varaiya_book} that agent $1$ can use the probability distribution 
\begin{gather} \label{pi_1_def}
    \Pi_t^1 := \mathbb{P}^{\boldsymbol{g}}\big(X_t, L_t^2 ~|~ M_t^1, \Gamma_{0:t-1}^2\big), \quad t=0,\dots,T,
\end{gather}
which takes values in the set of feasible distributions $\mathcal{P}^1_t := \Delta(\mathcal{X}_t \times \mathcal{L}_t^2)$, where $\Gamma_{0:t-1}^2$ are known given $\boldsymbol{\psi}^2$ and $M_t^1$. 
Next, we show that the information state $\Pi_t^1$ evolves independent of the choice of strategies $\boldsymbol{g}^1$ and $\boldsymbol{\psi}^2$.

\begin{lemma} \label{pi_1_evol}
For all $t = 0,\dots,T-1$, there exists a function $\tilde{f}_t^{{1}}(\cdot)$ independent of control strategy $\boldsymbol{g}^1$ and prescription strategy $\boldsymbol{\psi}^2$, such that
$\Pi_{t+1}^1 = \tilde{f}_{t}^{{1}}(\Pi_t^{{1}},U_t^1, \Gamma_t^2, Z_{t+1}^{{1}})$,
and subsequently, for any Borel subset $P^1 \subseteq \mathcal{P}^1_{t+1}$,
$\mathbb{P}(\Pi_{t+1}^1 \in P^1|M_t^1, U_{0:t}^1, \Gamma_{0:t}^2) = \mathbb{P}(\Pi_{t+1}^1 \in P^1|\Pi_t^1, U_t^1, \Gamma_t^2)$.
\end{lemma}

\begin{proof}
The proof follows the same arguments as the ones of Lemma \ref{pi_2_evol} in Section \ref{subsection:analysis_2}.
\end{proof}

\begin{lemma} \label{lem_pi_1_cost}
For any given prescription strategy $\boldsymbol{\psi}^2$ of agent $2$, there exists a function $\tilde{c}^1_t(\cdot)$ for all $t=0,\dots,T$, such that
\begin{align}
    \mathbb{E}^{\boldsymbol{g}}[c_t(X_t,U_t^{1:2}) ~ | ~ M_t^1,U_t^1,\Gamma_t^{2}] &= \tilde{c}^1_t(\Pi_t^1, A_t^2, U_t^1). \label{eq_lem_pi_cost}
\end{align}
\end{lemma}
\begin{proof}
The proof follows the same arguments as the ones of Lemma \ref{lem_pi_cost} in Section \ref{subsection:analysis_2}.
\end{proof}

The distribution $\Pi_t^1$ is called an \textit{information state} of agent $1$ at time $t$. As a consequence of Lemmas \ref{pi_1_evol} and \ref{lem_pi_1_cost}, the information state yields the following result for Problem \ref{problem_2}.

\begin{theorem} \label{pbp_struct_result}
    For any given prescription strategy $\boldsymbol{\psi}^{2}$ of agent $2$ in Problem \ref{problem_2}, without loss of optimality, we can restrict attention  to control strategies $\boldsymbol{g}^{*1}$ with the structural form
    \begin{gather} \label{eq_pbp_struct_result}
        U_t^{1} = g_t^{*1}(A_t^2, \Pi_t^1), \quad t = 0,\dots,T.
    \end{gather}
\end{theorem}

\begin{proof}
This proof follows standard arguments for centralized stochastic control problems in \cite[page 79]{varaiya_book}, and thus, it is omitted.
\end{proof}

Theorem \ref{pbp_struct_result} establishes a structural form for an optimal control strategy $\boldsymbol{g}^{*1}$ in Problem \ref{problem_2}, which holds for all $\boldsymbol{\psi}^2$, and subsequently, for all $\boldsymbol{g}^2$. From Remark \ref{remark_3}, we note that any optimal control strategy $(\boldsymbol{g}^{*1}, \boldsymbol{g}^{*2})$ for Problem \ref{problem_1} must yield a corresponding prescription strategy $\boldsymbol{\psi}^{*2}$ such that after fixing $\boldsymbol{\psi}^{*2}$, the control strategy $\boldsymbol{g}^{*1}$ is the optimal solution to Problem \ref{problem_2}. Thus, there exists an optimal control strategy $(\boldsymbol{g}^{*1}, \boldsymbol{g}^{*2})$ for Problem \ref{problem_1} where $\boldsymbol{g}^{*1}$ takes the structural form in \eqref{eq_pbp_struct_result}.


\begin{remark}
Consider that $|\mathcal{X}_t \times \mathcal{L}_t^2| = m \in \mathbb{N}$. Then, the information state $\Pi_t^1$ takes values in the continuous space $\mathcal{P}_t^1 = \big\{\big(p_t(1), \dots, p_t(m)\big) \in [0,1]^m: \sum_{i=1}^m p_t(i) = 1\big\}$. However, for all $t = 0,\dots,T$, the information state can only take \textit{countably} many realizations because all random variables take values in finite sets. For example, at $t=0$, for each $x_0 \in \mathcal{X}_0$ and $l_0^2 \in \mathcal{L}_0^2$, the probability $\mathbb{P}^{\boldsymbol{g}}(x_0, l_0^2~|~z_0^1)$ can take only finitely many values, i.e., one value for each $z_0^1 \in \mathcal{Z}_0^1$. Similarly, at any finite $t$, the memory $M_t^1$ can take finitely many realizations and thus, there are finitely many realizations for $\Pi_t^1$. As the horizon $T \to \infty$, the information state may take at most countably infinite realizations.
\end{remark}


\subsection{Analysis for Agent 2} \label{subsection:analysis_2}

In this subsection, we restrict agent $1$ to control strategies $\boldsymbol{g}^1$ which satisfy \eqref{eq_pbp_struct_result}, and derive a structural form for the optimal prescription strategy of agent $2$. Given $\boldsymbol{g}^1$, agent $2$ cannot generate the action $U_t^1$ at each $t$ because they cannot access $\Pi_t^1$. Thus, we consider a two stage process to generate the action of agent $1$: (1)  agent $2$ generates a prescription for agent $1$ using only $A_t^2$, and (2) agent $1$ computes $U_t^1$ using this prescription along with $\Pi_t^1$.

\begin{definition}
For all $t=0,\dots,T$, a \textit{prescription} for agent $1$ is a function $\Gamma_t^{1}: \mathcal{P}^{1}_t \to \mathcal{U}_t^1$ which takes values in a finite set $\mathcal{F}^{1}_t$.
\end{definition}

At each $t$, the prescription for agent $1$ is generated using a prescription law $\psi_t^{1}: \mathcal{A}_t^2 \to \mathcal{F}_t^{1}$, which yields $\Gamma_t^{1} = \psi_t^{1}(A_t^2)$. We call $\boldsymbol{\psi}^1 := \big(\psi_t^{1} : t=0,\dots,T\big)$ the prescription strategy of agent $1$ and $\boldsymbol{\psi} := (\boldsymbol{\psi}^1, \boldsymbol{\psi}^2)$ the prescription strategy of the system. For a given prescription $\Gamma_t^1$, agent $1$ computes their action as $U_t^1=\Gamma_t^1(\Pi_t^1)$. Next, we set up a new centralized problem from the perspective of agent $2$ with a state $S_t^2 := \{X_t, L_t^2, \Pi_t^{1}\}$ for all $t$, which takes values in the finite collection of sets ${\mathcal{S}}_t^2$. 
Moreover, we can construct a state evolution function $\bar{f}^2_t(\cdot)$ such that $S^2_{t+1} = \bar{f}^2_t(S^2_t, \Gamma_t^{1:2}, W_{t}, V_{t+1}^{1:2})$ and an observation rule $\bar{h}^2_t(\cdot)$ which yields $Z^2_{t+1} = \bar{h}^2_t(S^2_t, \Gamma_t^{1:2}, W_{t}, V_{t+1}^{1:2})$ for all $t=0,\dots,T-1$. Similarly, we can construct a cost function $\bar{c}^2_t(\cdot)$ such that $\bar{c}^2_t(S^2_t,\Gamma_t^{1:2}) := c_t(X_t, \Gamma_t^1(\Pi_t^1), \Gamma_t^2(L_t^2))$ for all $t$. Thus, the new centralized problem for agent $2$ has the state $S_t^2$, observation $Z_t^2$ and action $(\Gamma_t^1, \Gamma_t^2)$ at each $t$. The corresponding performance criterion is $\mathcal{J}^2(\boldsymbol{\psi}) = \mathbb{E}^{\boldsymbol{\psi}} [\sum_{t=0}^T\bar{c}_t^2(S_t^2, \Gamma_t^{1:2})]$. 

\begin{problem} \label{problem_3}
The optimization problem for agent $2$ is
    $\inf_{\boldsymbol{\psi}} \mathcal{J}^2(\boldsymbol{\psi})$,
given the probability distributions of the primitive random variables $\{X_0,W_{0:t},V_{0:t}^{1:2}\}$, and the dynamics $\{\bar{c}^2_t,\bar{f}^2_t,\bar{h}^2_t:t=0,\dots,T\}$.
\end{problem}

\begin{remark}
Using the same sequence of arguments as Lemma \ref{lem_psi_g_relation}, for each control strategy $\boldsymbol{g}$, we can construct an equivalent prescription strategy $\boldsymbol{\psi}$ such that $\mathcal{J}(\boldsymbol{g}) = \mathcal{J}(\boldsymbol{\psi})$ and vice versa. Thus, we always ensure that $\boldsymbol{\psi}$ is consistent with $\boldsymbol{g}$, which implies that for all $t$, $\Pi_t^1 = \mathbb{P}^{\boldsymbol{g}} (X_t ~|~ M_t^1, \Gamma_{0:t-1}^2) = \mathbb{P}^{\boldsymbol{\psi}} (X_t ~|~ M_t^1, \Gamma_{0:t-1}^2)
    = \mathbb{P}^{\boldsymbol{\psi}} (X_t ~|~ M_t^1, \Gamma_{0:t-1}^1, \Gamma_{0:t-1}^2)$,
where we can add $\Gamma_{0:t-1}^{1}$ to the conditioning because they are functions of $A_t^2 \subseteq M_t^1$ and $\boldsymbol{\psi}^1$. 
Because of this property, we can equivalently write the dependence of a probability distribution on either $\boldsymbol{g}$ or $\boldsymbol{\psi}$.
\end{remark}

Problem \ref{problem_3} is a partially observed centralized stochastic control problem and thus, agent $2$ must estimate the state $S_t^2$ at each time $t$. For this purpose, agent $2$ can use the distribution
\begin{gather} \label{pi_2_def}
    \Pi_t^2 := \mathbb{P}^{\boldsymbol{\psi}}(X_t, L_t^2, \Pi_t^1 ~|~A_t^2, \Gamma_{0:t-1}^{1:2}), \quad t =0,\dots,T,
\end{gather}
which takes values in the set of feasible distributions $\mathcal{P}_t^2 := \Delta(\mathcal{X}_t \times \mathcal{L}_t^2 \times \mathcal{P}_t^1)$. Recall that at each $t$, the information state of agent $1$, $\Pi_t^1$, can take at most countably infinitely many realizations in the space $\mathcal{P}^1_t$. Thus, the information state $\Pi_t^2$ can be represented using a tuple of probability mass functions $\big(p_t(x_t, \ell_t^2, \cdot~|a_t^2, \gamma_{0:t-1}^{1:2}) : x_t \in \mathcal{X}_t, \ell_t^2 \in \mathcal{L}_t^2\big)$, where $p_t(x_t, \ell_t^2, \cdot~|a_t^2, \gamma_{0:t-1}^{1:2}): \mathcal{P}_t^1 \to [0,1]$ for each $x_t \in \mathcal{X}_t$ and $\ell_t^2 \in \mathcal{L}_t^2$.
Next, we show that the evolution of $\Pi_t^2$ is Markovian and independent of the prescription strategy $\boldsymbol{\psi}$.


\begin{lemma} \label{pi_2_evol}
For all $t = 0,\dots,T-1$, there exists a function $\tilde{f}_t^{{2}}(\cdot)$ independent of the prescription strategy $\boldsymbol{\psi}$, such that
$\Pi_{t+1}^2 = \tilde{f}_t^2(\Pi_t^2,\Gamma_t^{1}, \Gamma_t^2,Z_{t+1}^2),$
and subsequently, for any Borel subset $P^2 \subseteq \mathcal{P}^2_{t+1}$,
$\mathbb{P}(\Pi_{t+1}^2 \in P^2~|~A_t^2, \Gamma_{0:t}^{1:2}) = \mathbb{P}(\Pi_{t+1}^2 \in P^2~|~\Pi_{t}^2, \Gamma_{t}^{1:2}).$
\end{lemma}

\begin{proof}
Let $x_t$, $\gamma_t^1$, $\gamma_t^2$, $a_t^2$, and $\pi_t^1$ be realizations of $X_t$, $\Gamma_t^1$, $\Gamma_t^2$, $A_t^2$, and the distribution $\Pi_t^1$, respectively, for all $t$.
Then, using Bayes' rule
\begin{multline} 
    \mathbb{P}^{\boldsymbol{\psi}}(x_{t+1}, \ell_{t+1}^2, \pi_{t+1}^1 ~|~ a_{t+1}^2, \gamma_{0:t}^{1:2}) \\
    = \dfrac{\mathbb{P}^{\boldsymbol{\psi}}\big(x_{t+1}, \ell_{t+1}^2, \pi_{t+1}^1, z_{t+1}^2 ~|~  a^2_t, \gamma_{0:t}^{1:2}\big)}{\mathbb{P}^{\boldsymbol{\psi}}\big(z_{t+1}^2 ~|~ a^2_t, \gamma_{0:t}^{1:2}\big)}, \label{lem_6_1}
\end{multline}
where $a_{t+1}^2 = a_t^2 \cup z_{t+1}^2$.
Using the dynamics $\{\bar{f}_t^2, \bar{h}_t^2, \bar{c}_t^2\}$, we write that
$(x_{t+1}, \ell_{t+1}^2)$ $= \eta_t^2(s_t^2, \gamma_t^{1:2}, w_t, v_{t+1}^{1:2})$, $\pi_{t+1}^{1} = \xi_t^2(s_t^2,\gamma_t^{1:2}, w_t, v_{t+1}^{1:2})$, $z_{t+1}^2 = \bar{h}_{t}^2(s_t^2, \gamma_t^{1:2}, w_t, v_{t+1}^{1:2})$,
for some appropriate functions $\eta_t^2(\cdot)$ and $\xi_t^2(\cdot)$, where $s_t^2 = \{x_t, \ell_t^2, \pi_t^1\}$. Substituting these relationships into the numerator in the RHS of \eqref{lem_6_1} yields that
\begin{multline}
    \mathbb{P}^{\boldsymbol{\psi}}\big(x_{t+1}, \ell_{t+1}^2, \pi_{t+1}^1, z_{t+1}^2 ~|~  a^2_t, \gamma_{0:t}^{1:2}\big) \\
    = \sum_{s_t^2, w_t, v_{t+1}^{1:2}} \mathbb{I}[\eta_t^2(s_t^2, \gamma_t^{1:2}, w_t) = (x_{t+1}, \ell_{t+1}^2)] \cdot \mathbb{P}(w_t, v_{t+1}^{1:2}) \\
    \cdot \mathbb{I}[\xi_t^2(s_t^2,\gamma_t^{1:2}, w_{t}, v_{t+1}^{1:2}) = \pi_{t+1}^1] \cdot \mathbb{P}^{\boldsymbol{\psi}}\big(s_t^2 ~|~ a^2_t, \gamma_{0:t-1}^{1:2}\big) \\
    \cdot \mathbb{I}[\bar{h}_{t+1}^2(s_t^2,\gamma_t^{1:2}, w_t, v_{t+1}^{1:2}) = z_{t+1}^2 ], \label{lem_6_3}
\end{multline}
where $\mathbb{I}(\cdot)$ is the indicator function, and where we can drop the prescriptions $\gamma_{t}^{1:2}$ from the conditioning in the last term because they are completely determined given $\boldsymbol{\psi}$ and $a_t^2$. Note that in \eqref{lem_6_3}, $\mathbb{P}^{\boldsymbol{\psi}}\big(s_t^2 ~|~ a^2_t, \gamma_{0:t-1}^{1:2}\big) = \pi_t^2(s_t^2)$. 
Next, we expand the denominator in \eqref{lem_6_1} as
\begin{multline}
    \mathbb{P}^{\boldsymbol{\psi}}\big(z_{t+1}^2 ~|~ a^2_t, \gamma_{0:t}^{1:2}\big) %
    = \sum_{s_t^2, w_t, v_{t+1}^{1:2}} \mathbb{P}(w_t, v_{t+1}^{1:2}) \\ \cdot \mathbb{I}[\bar{h}_{t}^2(s_t^2,\gamma_t^{1:2}, w_t, v_{t+1}^{1:2}) = z_{t+1}^2 ] \cdot \pi_t^2(s_t^2). \label{lem_6_4}
\end{multline}
Then, the first result holds by constructing an appropriate function $\tilde{f}_t^2(\cdot)$ using \eqref{lem_6_1} - \eqref{lem_6_4}. To prove the second result, for any Borel subset $P^2 \subseteq \mathcal{P}_{t+1}^2$, we write that
\begin{multline}
    \mathbb{P}(\Pi_{t+1}^2 \in P^2~|~ a_t^2, \gamma^{1:2}_{0:t},\pi^2_{0:t}) 
    =\sum_{z^2_{t+1}}\mathbb{I}[\tilde{f}_t^2(\pi_t^2,\gamma^{1:2}_{t}, \\
    z_{t+1}^2) \in P^2]
    \cdot\mathbb{P}(z^2_{t+1}~|~a_t^2,\gamma^{1:2}_{0:t},\pi_{0:t}^2)\label{proof_2_1}.
\end{multline}
The second term in \eqref{proof_2_1} can be expanded as
$\mathbb{P}(z^2_{t+1}~|~a_t^2,\gamma^{1:2}_{0:t},\pi_{0:t}^2)
    =\sum_{s_t^2, w_t, v_{t+1}^{1:2}}\mathbb{I}[\bar{h}_{t}^2(s_t^2,\gamma_t^{1:2},w_t,$ $  v_{t+1}^{1:2}) = z_{t+1}^2]
    \cdot\mathbb{P}(v_{t+1}^{{1:2}}, w_t)\cdot\pi_t^2(s_t^2)$.
The proof is complete by substituting this equation into \eqref{proof_2_1}.
\end{proof}

\begin{lemma} \label{lem_pi_cost}
There exists a function $\tilde{c}^2_t(\cdot)$ for all $t$, such that
\begin{align}
    \mathbb{E}^{\boldsymbol{g}}[c_t(X_t,U_t^{1:2}) ~ | ~ A_t^2,\Gamma_t^{1:2}] &= \tilde{c}^2_t(\Pi_t^2,\Gamma_t^{1:2}). \label{eq_lem_pi_cost_2}
\end{align}
\end{lemma}

\begin{proof}
Let $a_t^3$, $\gamma_t^{1:2}$, and $\pi_t^2$ be realizations of the random variables $A_t^3$, $\Gamma_t^{1:2}$, and the conditional distribution $\Pi_t^2$, respectively, for all $t = 0,\dots,T$. To prove the result, we expand the expectation as
$\mathbb{E}^{\boldsymbol{g}}[c_t(X_t,U_t^{1:2}) ~|~ a_t^2, \gamma_t^{1:2}] = \mathbb{E}^{\boldsymbol{\psi}}[\bar{c}_t^2(S^2_t,\Gamma_t^{1:2}) ~|~ a_t^2, \gamma_t^{1:2}]
    = \sum_{s^2_t} \bar{c}_t^2(s^2_t,\gamma_t^{1:2}) \cdot \mathbb{P}^{\boldsymbol{\psi}}(S^2_t = s^2_t ~|~ a_t^2,\gamma_t^{1:2})
    = \sum_{s^2_t} \bar{c}_t^2(s^2_t,\gamma_t^{1:2}) \cdot \pi_t^2(s^2_t) =: \tilde{c}^2_t(\pi_t^2,\gamma_t^{1:2}),$
where we can drop the prescripitions $\gamma_t^{1:2}$ from the conditioning because they known given $\boldsymbol{\psi}$ and $a_t^2$.
\end{proof}

We call $\Pi_t^2$ the information state of agent $2$ at time $t$. As a consequence of Lemmas \ref{pi_1_evol} and \ref{lem_pi_1_cost}, the information state yields the following result for Problem \ref{problem_3}.

\begin{theorem} \label{pbp_struct_result_2}
    In Problem \ref{problem_3}, without loss of optimality, we can restrict our attention  to prescription strategies $\boldsymbol{\psi}^{*}$ with the structural form
    \begin{gather} \label{eq_pbp_struct_result_2_1}
        \Gamma_t^{k} = \psi_t^{*k}\big(\Pi_t^2\big),\quad k = 1,2, \quad t = 0,\dots,T.
    \end{gather}
\end{theorem}

\begin{proof}
This proof follows similar arguments for centralized stochastic control problems in \cite[page 79]{varaiya_book}, and thus, it is omitted.
\end{proof}

Consider a prescription strategy $\boldsymbol{\psi}^* = (\boldsymbol{\psi}^{*1}, \boldsymbol{\psi}^{*2})$ which is an optimal solution to Problem \ref{problem_3}, and a control strategy $\boldsymbol{g}^{*}) = \boldsymbol{g}^{*1}, \boldsymbol{g}^{*2})$ given by $g_t^{*1}(\Pi_t^{1:2}) := \psi_t^{*1}(\Pi_t^2)(\Pi_t^1)$ and $g_t^{*2}(L_t^2, \Pi_t^2) := \psi_t^{*2}(\Pi_t^2)(L_t^2)$ for each $k = 1,2$ and $t = 0,\dots,T$. Using the same arguments as in Lemma \ref{lem_psi_g_relation}, we conclude that $\mathcal{J}(\boldsymbol{g}^{*}) = \mathcal{J}^2(\boldsymbol{\psi}^*)$ and subsequently, that $\boldsymbol{g}^*$ is the optimal solution to Problem \ref{problem_1}. Thus, without loss of optimality, we can restrict attention to control strategies $\boldsymbol{g}^*$ with the structural form $U_t^1 = g_t^{*1}(\Pi_t^1, \Pi_t^2)$ and $U_t^2 = g_t^{*2}(L_t^2, \Pi_t^2)$ for all $t=0,\dots,T$.

\begin{remark}
Consider a system where, the feasible sets of system variables are time invariant, i.e., $\mathcal{X}_t = \mathcal{X}$, $\mathcal{W}_t = \mathcal{W}$, $\mathcal{V}_t^k = \mathcal{V}^k$, $\mathcal{Y}_t^k = \mathcal{Y}^k$ for each $k = 1,2$ and $t=0,\dots,T$, and the information structure satisfies $\mathcal{L}_t^2 = \mathcal{L}^2$, $\mathcal{Z}_t^1 = \mathcal{Z}^1$, $\mathcal{Z}_t^2 = \mathcal{Z}^2$ for all $t$. Note that the set $\mathcal{M}_t^1$ still grows in size with time. However, the spaces $\mathcal{P}^1 = \Delta(\mathcal{X} \times \mathcal{L}^2)$ and $\mathcal{P}^2 = \Delta (\mathcal{X} \times \mathcal{L}^2 \times \mathcal{P}^1)$ are time invariant and subequently, our optimal control strategies have time-invariant domains for both agents. This is a useful property to derive and implement optimal control strategies for long time horizons. 
\end{remark}

\subsection{Dynamic Programming Decomposition} \label{subsection:DP}

In this subsection, we construct the value functions and corresponding control laws to form a DP decomposition which can derive the optimal prescription strategies.
Let $\gamma_t^k$ and $\pi_t^k$ be the realizations of the prescription $\Gamma_t^k$ and information state $\Pi_t^k$, respectively, for each $k = 1,2$ and $t=0,\dots,T$. Then, we recursively define the value functions
\begin{multline}
    J_t(\pi_t^2) := \inf_{\gamma_t^{1:2} \in \mathcal{F}_t^1 \times \mathcal{F}_t^2} \tilde{c}^2_t\big(\pi_t^2,\gamma_t^{1:2}\big) \\
    + \mathbb{E}^{\boldsymbol{\psi}}\Big[{J}_{t+1}\big(\tilde{f}_t^2(\pi_t^2, \gamma_t^{1:2}, Z_{t+1}^2) \big) ~|~\pi_t^{2}, \gamma_t^{1:2}\Big], \label{value_time_t}
\end{multline}
for all $t = 0,\dots,T$ and define $J_{T+1}(\pi_{T+1}^2):= 0$ identically. For each agent $k =1,2$, the prescription law at time $t$ is $\gamma_t^{*k} = \psi_t^{*k}(\pi_t^{2})$, i.e., the $\arg\inf$ in the RHS of \eqref{value_time_t}. 
The prescription strategy $\boldsymbol{\psi}^*$ derived using this DP decomposition can be shown to be the optimal solution to Problem \ref{problem_2} using standard arguments \cite{17, nayyar2019common}. Recall that given an optimal strategy $\boldsymbol{\psi}^*$ derived using this DP decomposition, we can also derive the optimal control strategy $\boldsymbol{g}^*$ for Problem \ref{problem_1}.


\begin{remark}
At each $t=0,\dots,T$, our DP decomposition requires solving an optimization problem for each realization $\pi_t^2$ of the information state $\Pi_t^2$, which is a tuple of probability mass functions. Optimizing over probability mass functions is a computationally challenging problem. Next, we present two different approaches to alleviate the computational implications. 
In Section \ref{section:decoupled}, we show how we can simplify our results when the system dynamics and information structure have additional favorable properties. In Section \ref{section:Implementation}, we present an approximation for the information states which can reduce the number of computations required to derive an approximately optimal strategy.
\end{remark}

\section{Simplification for Decoupled Dynamics} \label{section:decoupled}

In this subsection, we show how our results can be simplified when both agents have decoupled state and observation dynamics. We denote the state of each agent $k = 1,2$ at time $t$ by $X_t^k \in \mathcal{X}_t^k$. Starting at $X_0^k$, each state evolves as
\begin{align}
    X_{t+1}^k &= f_t^k(X_t^k, U_t^k, W_t^k), \quad t =0,\dots,T-1,
\end{align} 
for $k = 1,2$, where $W_t^k \in \mathcal{W}_t^k$ is a disturbance acting only on $X_t^k$. The observation of agent $k$ at time $t$ is
$Y_t^k = h_t^k(X_t^k, V_t^k)$. We assume that all primitive random variables $\{X_0^k, W_t^{k}, V_t^k: k = 1,2, \; t =0,\dots,T\}$ are independent of each other and that the cost to the system at each $t=0,\dots,T$ is $c_t(X_t^{1:2}, U_t^{1:2}) \in \mathbb{R}_{\geq0}$. Without loss of optimality, we restrict attention to control strategies where $\boldsymbol{g}^1$ takes the form $U_t^1 = g_t^1(\Pi_t^1, \Pi_t^2)$ and where $\boldsymbol{g}^2$ takes the form $U_t^2 = g_t^2(L_t^2, \Pi_t^2)$, for all $t=0,\dots,T$. Here, recall that $\Pi_t^1 = \mathbb{P}^{\boldsymbol{g}}\big(X^{1:2}_t,L_t^2|M_t^1, \Gamma_{0:t-1}^2\big)$ and $\Pi_t^2 = \mathbb{P}^{\boldsymbol{g}}\big(X^{1:2}_t,L_t^2,\Pi_t^1|A_t^2, \Gamma_{0:t-1}^{1:2}\big)$. Next, we show that the information state $\Pi_t^1$ can be simplified using the decoupled dynamics.



\begin{lemma} \label{lem_decouple_1}
    For each $k = 1,2$ and $t=0,\dots,T$, let $x_t^k$, $m_t^k$, $l_t^2$, and $a_t^2$ be realizations of the random variables $X_t^k$, $M_t^k$, $L_t^2$, and $A_t^2$, respectively. Then,
    \begin{align} 
        \mathbb{P}^{\boldsymbol{g}}(x_{t}^{1:2}, l_t^2~|~m_{t}^1) &= \mathbb{P}^{\boldsymbol{g}}(x_{t}^1 ~|~m_{t}^1) \cdot \mathbb{P}^{\boldsymbol{g}}(x_{t}^2, l_t^2~|~a_{t}^2). \label{decouple_1}
    \end{align}
\end{lemma}

\begin{proof}
Given the realizations $x_t^k$, $y_t^k$, $u_t^k$, $\gamma_t^k$ and $l_t^2$ of $X_t^k$, $Y_t^k$, $U_t^k$, $\Gamma_t^k$, and $L_t^2$, respectively, for each $k = 1,2$ and $t=0,\dots,T$, we prove \eqref{decouple_1} by mathematical induction. At $t=0$, depending on the information sharing pattern of the system, there are two possible realizations of the memory of agent $1$, either $m_0^1 = \{y_0^1\}$ or $m_0^1 = \{y_0^{1},y_0^2\}$. For the first realization of the memory of agent $1$, the private information of agent $2$ is $l_0^2 = \{y_0^2\}$, and thus, we can expand the LHS of \eqref{decouple_1} as
$\mathbb{P}^{\boldsymbol{g}}(x_0^{1:2},y_0^2|m_0^1) = \mathbb{P}^{\boldsymbol{g}}(x_0^{1:2},y_0^2|y_0^{1}) = \mathbb{P}^{\boldsymbol{g}}(x_0^1|y_0^{1}) \cdot \mathbb{P}^{\boldsymbol{g}}(x_0^2,y_0^2)$, where recall that the observation $y_0^k$ depends only on $x_0^k$ for each $k$, and the primitive random variables are independent of each other. For the second realization of the memory of agent $1$, note that $l_t^2 = \emptyset$ because $l_t^2 \cap m_t^1 = \emptyset$, and thus, we can expand the LHS as $\mathbb{P}^{\boldsymbol{g}}(x_0^{1:2}|m_0^1) = \mathbb{P}^{\boldsymbol{g}}(x_0^1|y_0^{1})\cdot \mathbb{P}^{\boldsymbol{g}}(x_0^2|y_0^{2})$. For both cases at $t=0$, we have shown the LHS is equal to the RHS in \eqref{decouple_1}. This forms the basis of our induction. Next, we consider the induction hypothesis that \eqref{decouple_1} holds at each $0,\dots,t$, and expand the LHS at $t+1$ as
    \begin{multline} \label{proof_7_1}
        \mathbb{P}^{\boldsymbol{g}}(x_{t+1}^{1:2}, l_{t+1}^2~|~m_{t+1}^1) = \dfrac{\mathbb{P}^{\boldsymbol{g}}(x_{t+1}^{1:2}, l^{2}_{t+1}, z_{t+1}^1~|~m_t^1)}{\mathbb{P}^{\boldsymbol{g}}(z_{t+1}^1~|~m_t^1)} \\
        = \dfrac{\mathbb{P}^{\boldsymbol{g}}(x_{t+1}^{1:2}, l^{2}_{t+1}, z_{t+1}^1~|~m_t^1)}{\sum_{x_{t+1}^{1:2},l_{t+1}^2}\mathbb{P}^{\boldsymbol{g}}(x_{t+1}^{1:2}, l^{2}_{t+1}, z_{t+1}^1~|~m_t^1)}.
    \end{multline}
    Note that in the partially accessible information structure, $l_{t+1}^2 \cup z_{t+1}^1 = l_t^2 \cup \{y_{t+1}^{1:2}, u_t^{1:2}\}$. Thus, we can write that $\mathbb{P}^{\boldsymbol{g}}(x_{t+1}^{1:2}, l^{2}_{t+1}, z_{t+1}^1~|~m_t^1) = \mathbb{P}^{\boldsymbol{g}}(x_{t+1}^{1:2}, y_{t+1}^{1:2}, u_t^{1:2}, l^{2}_{t}~|~m_t^1) = \mathbb{P}^{\boldsymbol{g}}(y^1_{t+1}|x_{t+1}^1) \cdot \mathbb{P}^{\boldsymbol{g}}(y^2_{t+1}|x_{t+1}^2) \cdot \mathbb{I}[g_t^1(m_t^1) = u_t^1]$ $\cdot \mathbb{I}[\gamma_t^2(l_t^2) = u_t^2] \cdot \mathbb{P}^{\boldsymbol{g}}(x_{t+1}^{1:2}, l_t^2|m_t^1)$, where $\mathbb{I}(\cdot)$ is the indicator function, and where $\gamma_t^2$ and $u_t^1$ are completely determined given $m_t^1$ and $\boldsymbol{g}$. Furthermore, we expand the last term as $\mathbb{P}^{\boldsymbol{g}}(x_{t+1}^{1:2},l^2_{t}|m_t^1, u_t^1, \gamma_t^2) = \sum_{x_t^{1:2}, w_t^{1:2}} \mathbb{I}[f_t^1(x_t^1, u_t^1, w_t^1) = x_{t+1}^1] \cdot \mathbb{I}[f_t^2(x_t^2, \gamma_t^2(l_t^2), w_t^2) = x_{t+1}^2]\cdot \mathbb{P}(w_t^{1},w_t^2) \cdot \mathbb{P}^{\boldsymbol{g}}(x_t^{1:2},l_t^2|m_t^1)$, where we can use the induction hypothesis to obtain $\mathbb{P}^{\boldsymbol{g}}(x_t^{1:2},l_t^2|m_t^1) = \mathbb{P}^{\boldsymbol{g}}(x_t^{1}|m_t^1)\cdot \mathbb{P}^{\boldsymbol{g}}(x_t^{1:2},l_t^2|a_t^2)$.
    Substituting these results into \eqref{proof_7_1}, and rearranging the terms yields
    $\mathbb{P}^{\boldsymbol{g}}(x_{t+1}^{1:2},l^{2}_{t+1}|m_{t+1}^1)
    = \frac{\mathbb{P}^{\boldsymbol{g}}(x_{t+1}^1, y^1_{t+1}, u^1_t|m_t^1)}{\mathbb{P}^{\boldsymbol{g}}(y_{t+1}^1, u^1_t|m_t^1)} \cdot \mathbb{P}^{\boldsymbol{g}}(x_{t+1}^2, l^2_{t+1}|a_t^2,z_{t+1}^2)
    = \mathbb{P}^{\boldsymbol{g}}(x_{t+1}^1|$ $m_t^1, y_{t+1}^1, u^1_t) \cdot \mathbb{P}^{\boldsymbol{g}}(x_{t+1}^2,y^2_{t+1}|a_{t+1}^2).$
    To complete the proof by mathematical induction, we need to show that the first term in the RHS of the previous equation is equal to the first term in the RHS of \eqref{decouple_1}. We achieve this by expanding $\mathbb{P}^{\boldsymbol{g}}(x_{t+1}^1|m_t^1, y_{t+1}^{1:2}, u_t^{1:2}, l_t^2)$ $= \frac{\mathbb{P}^{\boldsymbol{g}}(x_{t+1}^1,y_{t+1}^1|m_t^1, l_t^2, y_{t+1}^2)}{\sum_{x_{t+1}^1} \mathbb{P}^{\boldsymbol{g}}(x_{t+1}^1,y_{t+1}^1|m_t^1, l_t^2, y_{t+1}^2)}$ $= \frac{\sum_{x_t^1}\mathbb{P}^{\boldsymbol{g}}(y_{t+1}^1|x_{t+1}^1)\cdot\mathbb{P}^{\boldsymbol{g}}(x_{t+1}^1|x_t^1, u_t^1)\cdot \mathbb{P}^{\boldsymbol{g}}(x_t^1|m_t^1)}{\sum_{x_{t:t+1}^1} \mathbb{P}^{\boldsymbol{g}}(y_{t+1}^1|x_{t+1}^1)\cdot\mathbb{P}^{\boldsymbol{g}}(x_{t+1}^1|x_t^1, u_t^1)\cdot \mathbb{P}^{\boldsymbol{g}}(x_t^1|m_t^1)}$, where, in the last equality, we use Bayes' rule and the induction hypothesis. Recall that $z_{t+1}^1 \subseteq l_t^2 \cup \{y_{t+1}^{1:2}, u_t^{1:2}\}$. This implies that $\mathbb{P}^{\boldsymbol{g}}(x_{t+1}^1|m_t^1, y_{t+1}^{1:2}, u_t^{1:2}, l_t^2) = \mathbb{P}^{\boldsymbol{g}}(x_{t+1}^1|m_t^1, z_{t+1}^{1}) =\mathbb{P}^{\boldsymbol{g}}(x_{t+1}^1|m_t^1, y_{t+1}^{1}, u_t^{1})$, which complete the proof.
\end{proof}

Motivated by Lemma \ref{lem_decouple_1}, we define the distributions $\Theta_t^1 := \mathbb{P}^{\boldsymbol{g}}(X_t^1|M_t^1)$ and $\Theta_t^2 := \mathbb{P}^{\boldsymbol{g}}(X_t^2,L_t^2|A_t^2)$ and note that the information state $\Pi_t^1$ at each $t =0,\dots,T$ can be written as a function of $(\Theta_t^1, \Theta_t^2)$. Thus, at time $t$, agent $1$ can track the distributions $(\Theta_t^1, \Theta_t^2)$ instead of $\Pi_t^1$ to compute their optimal control action $U_t^1$. Next, we show that the evolution of $\Theta_t^k$, for each $k = 1,2$, is Markovian, strategy independent and decoupled from the dynamics of the other agent.

\begin{lemma}
    At each time $t$, there exists a function $\tilde{e}_t^k(\cdot)$, independent of the strategy $\boldsymbol{g}$, for all $k = 1,2$ such that
    \begin{gather}
        \Theta_{t+1}^k = \tilde{e}_t^k(\Theta_t^k, U_t^k, Y_{t+1}^k).
    \end{gather}
\end{lemma}

\begin{proof}
The proof follows the same arguments as the ones in Lemma \ref{pi_2_evol} and thus, due to space limitations, it is omitted.
\end{proof}

Note that the distribution $\Theta_t^2$ is also available to agent $2$ at each $t=0,\dots,T$, because it depends only on the accessible information $A_t^2$. Subsequently, using the same sequence of arguments as the ones in Theorem \ref{pbp_struct_result_2}, we conclude that, without loss of optimality, agent $2$ can restrict attention to prescription strategies with the structural form $\Gamma_t^k = \psi_t^k\big(\mathbb{P}^{\boldsymbol{g}}(X_t^{1:2},L_t^2,\Theta_t^1|A_t^2), \Theta_t^2\big)$, for each $k =1,2$ and $t=0,\dots,T$. Next, we show that the term $\mathbb{P}^{\boldsymbol{g}}(X_t^{1:2},L_t^2,\Theta_t^1|A_t^2)$ in the argument of the prescription law for each $k$ can also be simplified using the decoupled dynamics of the system.

\begin{lemma} \label{lem_decouple_2}
    For each $k = 1,2$ and $t=0,\dots,T$, let $x_t^k$, $l_t^2$, $a_t^2$, and $\theta_t^k$ be realizations of the random variables $X_t^k$, $L_t^2$, $A_t^2$, and the probability distribution $\Theta_t^k$, respectively. Then,
    \begin{align}
        \mathbb{P}^{\boldsymbol{g}}(x_t^{1:2},l_t^2,\theta_t^1|a_t^2) = \mathbb{P}^{\boldsymbol{g}}(x_t^{1},\theta_t^1|a_t^2) \cdot \mathbb{P}^{\boldsymbol{g}}(x_t^{2},l_t^2|a_t^2).
    \end{align}
\end{lemma}

\begin{proof}
The proof follows by mathematical induction using the same arguments as the ones in Lemma \ref{lem_decouple_1}, and thus, due to space limitations, it is omitted.
\end{proof}

Starting with the structural form of optimal prescription strategies in Theorem \ref{pbp_struct_result_2}, we can use Lemmas \ref{lem_decouple_1} and \ref{lem_decouple_2}, to conclude that in systems with decoupled dynamics, without loss of optimality, we can restrict attention to control strategies $\boldsymbol{g}^*$ with the structural form
\begin{align}
    U_t^1 &= g_t^{*1}\big[\Theta_t^1, \Theta_t^2, \mathbb{P}^{\boldsymbol{g}}(X_t^1, \Theta_t^1~|~A_t^2)\big], \label{decoupled_result_1} \\
    U_t^2 &= g_t^{*2}\big[L_t^2, \Theta_t^2, \mathbb{P}^{\boldsymbol{g}}(X_t^1, \Theta_t^1~|~A_t^2)\big], \; t = 0,\dots,T. \label{decoupled_result_2}
\end{align}

\begin{remark}
The control strategy $\boldsymbol{g}^1$ yielded a control law for each $t=0,\dots,T$ for agent $1$ with the form $U_t^1 = g_t^1(\Pi_t^1, \Pi_t^2)$, which has the domain $\Delta(\mathcal{X}_t^1 \times \mathcal{X}_t^2 \times \mathcal{L}_t^2) \times \Delta(\mathcal{X}_t^1 \times \mathcal{X}_t^2 \times \mathcal{L}_t^2 \times \Delta(\mathcal{X}_t^1 \times \mathcal{X}_t^2 \times \mathcal{L}_t^2))$. In contrast, the domain of the control law $g_t^{*1}$ in \eqref{decoupled_result_1} is $\Delta(\mathcal{X}^1_t) \times \Delta(\mathcal{X}^2_t \times \mathcal{L}_t^2) \times \Delta(\mathcal{X}^1_t \times \Delta(\mathcal{X}^1_t))$, which is a space with a smaller dimension than the one before. Similarly, the control laws of agent $2$ have a domain with a smaller dimension in \eqref{decoupled_result_2} than the control laws derived using Theorem \ref{pbp_struct_result_2}. Thus, we have obtained a simpler form for an optimal control strategy in systems with decoupled dynamics.
\end{remark}

We can further simplify the structural form of the optimal control strategies when agent $1$ can perfectly observe the state $X_t^1$, i.e, $Y_t^1 = X_t$ and subsequently, $X_t^1 \subseteq M_t^1$ at each $t=0,\dots,T$. Then, for a given realization $m_t^1$ of the memory $M_t^1$, the probability distribution $\Theta_t^1$ at each $t$ is simply given by $\Theta_t^1 = \mathbb{I}[X_t^1 = x_t^1]$ for the realization $x_t^1 \in m_t^1$ of $X_t^1$, where $\mathbb{I}$ is the indicator function. Using this result in \eqref{decoupled_result_1} and \eqref{decoupled_result_2}, we conclude that, without loss of optimality, we can restrict attention to control strategies $\boldsymbol{g}^*$ with the form
\begin{align}
    U_t^1 &= g_t^1\big[X_t^1, \Theta_t^2, \mathbb{P}^{\boldsymbol{g}}(X_t^1~|~A_t^2)\big], \\
    U_t^2 &= g_t^2\big[L_t^2,\Theta_t^2, \mathbb{P}^{\boldsymbol{g}}(X_t^1~|~A_t^2)\big], \quad t = 0,\dots,T.
\end{align}

\begin{remark}
When agent $1$ can perfectly observe their own state, at each $t$, the domains of the optimal control laws $g_t^{*1}$ and $g_t^{*2}$ are $\mathcal{X}_t^1 \times \Delta(\mathcal{X}_t^2\times \mathcal{L}_t^2) \times \Delta(\mathcal{X}_t^1)$ and $\mathcal{L}_t^2 \times \Delta(\mathcal{X}_t^2 \times \mathcal{L}_t^2) \times \Delta(\mathcal{X}_t^1)$, respectively. These domains are small enough that the optimal control laws at each $t$ are functions of distributions over finite sets instead of probability mass functions. Thus, the resulting DP can be solved using standard techniques for centralized problems.
\end{remark}

\section{Implementation} \label{section:Implementation}


In this subsection, we present an approach to approximate the information state $\Pi_t^1$ for all $t =0,\dots,T$ which ensures that the approximation can only take finitely many values. To simplify the notation, we restrict our attention to systems where $|\mathcal{X}_t \times \mathcal{L}_t^2| = m$, $m \in \mathbb{N}$ for all $t = 0,\dots,T$. Furthermore, we consider that the maximum cost at each $t$ is bounded above by $||c||_{\infty} < \infty$.
Recall that the space of feasible values for $\Pi_t^1$ is the simplex $\mathcal{P}^1 = \left\{\big(p(1), \dots, p(m)\big) \in [0,1]^m : \sum_{i = 1}^m p(i) = 1\right\}$.
We use the procedure in \cite{reznik2011algorithm} to generate a set of equally distributed points in $\mathcal{P}^1$. Specifically, we select a number $n \in \mathbb{N}$ and define a set $\mathcal{Q}_n := \big\{ \big(q(1), \dots, q(m)\big) \in \mathcal{P}^1 : n \cdot q(i)$ $\in\mathbb{N}_{\geq0}, i = 1,\dots,m \big\}$. The set $\mathcal{Q}_n$ forms a lattice containing $|\mathcal{Q}_n| = \binom{m+n-1}{m-1}$ points in the simplex $\mathcal{P}^1$.
For example, let $\mathcal{X}_t = \{0,1\}$ and $\mathcal{L}_t^2 = \emptyset$, which implies that $m = 2$. Then, by selecting $n = 2$ we construct the set $\mathcal{Q}_2 = \big\{(0,1), (1/2, 1/2), (1,0)\big\}$. Similarly, if $m = 3$ and we select $n = 2$, we construct the set $\mathcal{Q}_2 = \big\{(1,0,0),$ $(1/2,1/2,0), (0,1,0), (0, 1/2, 1/2), (0,0,1), (1/2, 0, 1/2)\big\}$. Next, we define the total variation distance between any point in $\mathcal{P}^1$ and $\mathcal{Q}_n$, and then, we use this metric to define an approximate information state.

\begin{definition}
The total variation distance between any $\pi_t^1 = (p(1), \dots, p(m)) \in \mathcal{P}^1$ and any $q_t = (q(1), \dots, q(m)) \in \mathcal{Q}_n$ is $|{\pi}_t^1 - q|_{TV} = \sum_{i=1}^{m} |p(i) - q(i)|$.
\end{definition}

\begin{definition}
The approximate information state for agent $1$ at each $t=0,\dots,T$ is a random variable $\hat{\Pi}_t^1$ which takes values in the finite set $\mathcal{Q}_n$, and which is given by
\begin{gather}
   \hat{\Pi}_t^1 =  \sigma(\Pi^1_t) := \arg \min_{q \in \mathcal{Q}_n} |\Pi_t^1 - q|_{TV}. \label{sigma_def}
\end{gather}
\end{definition}
Given any distribution $\pi_t^1 \in \mathcal{P}^1$, the corresponding realization of the approximate information state $\hat{\pi}_t^1 = \sigma(\pi_t^1)$ can be efficiently computed using the algorithm in \cite{reznik2011algorithm}. Next, we present an upper bound in the total variation distance between any information state and its approximation.
 
\begin{lemma} \label{lem_bound}
For all $t=0,\dots,T$, for any realization $\pi_t^1$ of the information state $\Pi_t^1$, it holds that $|{\pi}_t^1 - \sigma(\pi_t^1)|_{TV} \leq \frac{2 a \cdot (1+a)}{m \cdot n}$, where $a = \lfloor m/2 \rfloor \in \mathbb{N}$ and $\lfloor \cdot \rfloor$ is the floor function. 

\end{lemma}

\begin{proof}
The proof follows from \cite[Proposition 2]{reznik2011algorithm}.
\end{proof}


Given any upper bound $\epsilon \in \mathbb{R}_{>0}$, we can use Lemma \ref{lem_bound} to construct a set $\mathcal{Q}_n$ which satisfies $|{\pi}_t^1 - \sigma(\pi_t^1)| \leq \epsilon$ for all $\pi_t^1 \in \mathcal{P}^1$, by selecting $n \geq \frac{2 a(1+a)}{ m\cdot \epsilon}$. Furthermore, the resulting approximate information state $\hat{\Pi}_t^1$ can be updated in a Markovian and strategy independent manner as
$\hat{\Pi}_{t+1}^1 = \sigma[\tilde{f}_t^1(\hat{\Pi}_t^1, U_t^1, \Gamma_t^2, Z_{t+1}^1)]$, for all $t=0,\dots,T-1$.

Our aim is to solve the centralized Problem \ref{problem_2} for agent $1$ using the approximate information state $\hat{\Pi}_t^1$, which takes only finitely many values for all $t = 0,\dots,T$, instead of the information state $\Pi_t^1$, which can take countably infinitely many values. For a fixed prescription strategy $\boldsymbol{\psi}^2$, recall from Lemma \ref{lem_pi_1_cost} that the expected cost at time $t$ can be written as $\tilde{c}_t^1(\Pi_t^1,A_t^2,U_t^1)$. Then, in Problem \ref{problem_2}, we can optimize the performance criterion $\mathcal{J}^1(\boldsymbol{g}^{1})$ using a centralized DP as follows. Let $u_t^1$, $a_t^2$, and $\pi_t^1$ be the realizations of $U_t^1$, $A_t^2$ and $\Pi_t^1$, respectively. Then, we define the value functions
\begin{multline}
    J_t^1(\pi_t^1, a_t^2) := \inf_{u_t^1 \in \mathcal{U}_t^1} \tilde{c}_t^1(\pi_t^1, a_t^2, u_t^1) \\
    + \mathbb{E}[J_{t+1}^1(\Pi_{t+1}^1, A_{t+1}^2)~|~\pi_t^1, a_t^2, u_t^1], \quad t = 0,\dots,T,  \label{pbp_DP_exact}
\end{multline}
and ${J}_{T+1}^1(\pi_{T+1},a_{T+1}) := 0$ identically. The person-by-person optimal control law at time $t$ is $u_t^{*1} = g_t^{*1}(\pi_t^1, a_t^2)$, i.e., the $\arg\inf$ in the RHS of \eqref{pbp_DP_exact}, and the performance of the system is $\mathcal{J}^1(\boldsymbol{g}^{*1}) = \mathbb{E}[J_0^1(\Pi_0^1, A_0^2)]$.

However, we seek the best control strategy $\boldsymbol{\hat{g}}^{*1}$ for Problem \ref{problem_2} which takes the structural form $u_t^1 = \hat{g}_t^1(\hat{\pi}_t^1, a_t^2)$ for all $t=0,\dots,T$.
Thus, we define the modified value functions
\begin{multline} 
    \hat{J}_t^1(\hat{\pi}_t^1, a_t^2) := \inf_{u_t^{1} \in \mathcal{U}^1} \tilde{c}_t^1(\hat{\pi}_t^1, a_t^2, u_t^{1}) \\
    + \mathbb{E}[\hat{J}_{t+1}^1(\hat{\Pi}_{t+1}^1, A_{t+1}^2)~|~\hat{\pi}_t^1, a_t^2, u_t^1], \quad t = 0,\dots,T, \label{pbp_DP_approximate}
\end{multline}
and $\hat{J}_{t+1}^1(\hat{\pi}_t^1, a_t^2) := 0$ identically, where $\hat{\pi}_t^1 = \sigma(\pi_t^1)$. For a fixed $\boldsymbol{\psi}^2$, the best control law using the approximate information state at each $t$ is $u_t^{*1} = \hat{g}_t^{*1}(\hat{\pi}_t^1, a_t^2)$, i.e., the $\arg\inf$ in the RHS of \eqref{pbp_DP_approximate}, and the performance of the system is $\mathcal{J}^1(\boldsymbol{\hat{g}}^{*1}) = \mathbb{E}[\hat{J}_0^1(\hat{\Pi}_0^1, A_0^2)]$. The \textit{loss in person-by-person performance} which arises from using the approximate information state is measured by the difference $|\mathcal{J}^1(\boldsymbol{g}^{*1}) - \mathcal{J}^1(\boldsymbol{\hat{g}}^{*1})|$. Next, we present a result for this performance loss.

\begin{lemma} \label{lem_asymptotic}
For any given prescription strategy $\boldsymbol{\psi}^2$,
\begin{gather}
    \lim_{n \to \infty} |\mathcal{J}^1(\boldsymbol{g}^{*1}) - \mathcal{J}^1(\boldsymbol{\hat{g}}^{*1})| = 0.
\end{gather}
\end{lemma}

\begin{proof}
The proof follows directly from \cite[Theorem 3]{saldi2019asymptotic}.
\end{proof}

Lemma \ref{lem_asymptotic} establishes the asymptotic convergence of the optimal performance by using the approximate information state towards the exact person-by-person optimal performance. Furthermore, it implies that for any desired upper bound on loss $\alpha_0 \in \mathbb{R}_{\geq0}$, there exists a number $n \in \mathbb{N}$ and set $\mathcal{Q}_n$, such that $|\mathcal{J}^1(\boldsymbol{g}^{*1}) - \mathcal{J}^1(\boldsymbol{\hat{g}}^{*1})| < \alpha_0$. An explicit relationship between the upper bound $\alpha_0$ and the upper bound on total variation distance, $\epsilon$ can be obtained using Theorem 9 and Proposition 46 of \cite{subramanian2022approximate}. This is given by recursively defining
\begin{gather}
    \alpha_t = 2(\epsilon \cdot ||c||_{\infty} + 3 \epsilon \cdot ||\hat{J}^1_{t+1}||_{\infty} + 3 \epsilon \cdot \hat{J}_{L}^1 + \alpha_{t+1}),
\end{gather}
where $||\hat{J}^1_{t+1}||_{\infty} := \sup_{\hat{\pi}_t^1, a_t^2}$ $\hat{J}_{t+1}^1(\hat{\pi}_{t+1}^1, a_{t+1}^2)$ and $\hat{J}_L^1$ is a finite upper bound on the Lipschitz constant of $\hat{J}_t^1$ for all $t = 0,\dots,T$. Note that an upper bound on the value of $\hat{J}_{t}$ exists for all $t=0,\dots,T$ because cost is upper bounded. Furthermore, the Lipschitz continuity of $\hat{J}_t^1$ arises naturally from the fact that it is piece-wise linear and concave with respect to $\hat{\pi}_t^1$ for all $t=0,\dots,T$ \cite{smallwood1973optimal}. 




The maximum loss in person-by-person performance from using an approximate information state in $\mathcal{Q}_n$ is $||\alpha_0||_{\infty} := \sup_{\boldsymbol{\psi}^2} \alpha_0$. Furthermore, we define an approximate information state for agent $2$ as $\hat{\Pi}_t^2 := \mathbb{P}^{\boldsymbol{\psi}}(X_t^1, L_t^2, \hat{\Pi}_t^1~|~M_t^2, \Gamma_{0:t-1}^{1:2})$. In a manner similar to Lemma \ref{pi_2_evol}, we can show that at each $t=0,\dots,T-1$, there exists a function $\hat{f}_t^2$ such that $\hat{\Pi}_{t+1}^2 = \hat{f}_t^2(\hat{\Pi}_t^2, \Gamma_t^{1:2}, Z_{t+1}^2)$. Thus, using the same sequence of arguments as Theorem \ref{pbp_struct_result_2}, we conclude that if we restrict our attention to control strategies with the structural form $U_t^1 = \hat{g}_t^{1}(\hat{\Pi}_t^1, \hat{\Pi}_t^2)$, and $U_t^2 = \hat{g}_t^{2}(L_t^2, \hat{\Pi}_t^2)$ for all $t =0,\dots,T$, the maximum loss in optimal performance in Problem \ref{problem_1}, $|\mathcal{J}(\boldsymbol{g}^*) - \mathcal{J}(\boldsymbol{\hat{g}}^{*1}, \boldsymbol{\hat{g}}^{*1})|$, is also $||\alpha_0||_{\infty}$.

\begin{remark}
In this approximation technique, the set of feasible values of $\hat{\Pi}_t^1$, $\mathcal{Q}_n$, is finite and does not grow in size with time. Thus, $\hat{\Pi}_t^2$ is a simple probability distribution with a finite support, which, in turn, simplifies the implementation of our DP. However, it is still challenging to compute globally optimal prescription strategies for moderate and large values of the parameter $n \in \mathbb{N}$ because the number of possible prescriptions of agent $1$, $|\mathcal{U}_t^1|^{|\mathcal{Q}_n|}$, grows exponentially with $n$. Instead, this approach may be utilized when only person-by-person optimal strategies are required. 
\end{remark}

\section{Conclusions}

In this paper, we introduced a general model for decentralized control of two agents with nested accessible information. We derived structural forms for optimal control strategies with domains which do not grow in size with time and thus, can be derived using a DP decomposition. We also presented simplified optimal control strategies for systems with decoupled state and observation dynamics. Finally, we presented an approximate information state which can be used to derive approximately optimal control strategies with smaller domains. One potential direction for future research includes deriving more efficient approximate representations of the information state and prescriptions. Another important direction of future research is the development of approximate algorithms specialized to efficiently solve the DPs which arise in decentralized control.

\bibliographystyle{ieeetr}
\bibliography{References}

\end{document}